\documentclass{amsart}
\usepackage{graphicx}
\usepackage{amssymb, amsmath}
\usepackage{graphicx}
\usepackage{tikz}
\usepackage{cleveref}
\usepackage{bm}
\usepackage{xfrac}

\newtheorem{theorem}{Theorem}[section]
\newtheorem{corollary}[theorem]{Corollary}

\theoremstyle{definition}
\newtheorem{definition}[theorem]{Definition}

\newtheorem{example}[theorem]{Example}

\theoremstyle{remark}

\numberwithin{equation}{section}

\title[Magnitude and nonnegative weightings]{Continuity of the magnitude for finite metric spaces with nonnegative weightings}

\author{Yuki Hiyoshi}
\address{Department of Mathematics, The University of Osaka, Japan}
\email{u565398f@ecs.osaka-u.ac.jp}

\begin{document}

\begin{abstract}
For finite metric spaces, magnitude is an invariant known to be discontinuous everywhere with respect to the Gromov-Hausdorff distance. While much of the literature focuses on positive definite metric spaces, here we consider finite metric spaces admitting a nonnegative weighting. In this paper, we focus on the class of finite metric spaces that admit a nonnegative weighting. This class includes ultrametric spaces, for which specific bounds on magnitude are known. We generalize these bounds by proving that, for any two finite metric spaces with nonnegative weightings, the ratio of their magnitudes is exponentially bounded by their Gromov-Hausdorff distance. Consequently, we show that within this class, the logarithmic magnitude function is Lipschitz continuous, and the magnitude function itself is locally Lipschitz continuous, but not Lipschitz continuous.
\end{abstract}

\maketitle

\setcounter{tocdepth}{1}


\section{Introduction}

The concept of magnitude for metric spaces, introduced by Leinster \cite{Leinster2013}, is a fundamental invariant in metric geometry. Let $(X, d_X)$ be a finite metric space. The zeta matrix of $X$ is defined as $Z_X = (e^{-d_X(x, y)})_{x, y \in X}$. A weighting on $X$ is a vector $w_X = (w_x)_{x \in X} \in \mathbb{R}^X$ satisfying $Z_X w_X = \mathbf{1}$, where $\mathbf{1}$ is the column vector with all entries equal to $1$. If a weighting exists, the magnitude of $X$ is defined as the sum of its weights, $\mathrm{Mag}(X) = \sum_{x \in X} w_x$. A weighting $w_X$ is said to be nonnegative if $w_x \geq 0$ for every $x \in X$. 

To discuss the continuity of magnitude, it is necessary to consider the distances between metric spaces \cite{bh-non}. Let $Z$ be a metric space, and let $A, B$ be subsets of $Z$. The Hausdorff distance between $A$ and $B$, denoted by $d_H(A, B)$, is defined as the infimum of all $\varepsilon > 0$ such that $A$ is contained in the $\varepsilon$-neighborhood of $B$, and $B$ is contained in the $\varepsilon$-neighborhood of $A$. Based on this, the Gromov-Hausdorff distance $d_{GH}(X, Y)$ between two abstract metric spaces $X$ and $Y$ is defined as the infimum of the Hausdorff distances $d_H^d(X, Y)$ taken over all possible metrics $d$ on the disjoint union $X \sqcup Y$ that restrict to the given metrics $d_X$ and $d_Y$ on $X$ and $Y$, respectively.

Regarding continuity with respect to the Gromov-Hausdorff distance, it is known that the magnitude function is discontinuous everywhere\cite[Thm.2.5]{kry-con}. An important class on which magnitude has good properties is the class of positive definite metric spaces. Here, a finite metric space $X$ is called positive definite if its zeta matrix $Z_X$ is positive definite. See, for example, \cite{Meckes2013}. Positive definite metric spaces have several useful properties, but nonnegative weightings also arise naturally beyond the positive definite setting. An example is provided by ultrametric spaces. Although positive definite spaces exhibit many nice analytic properties, another important direction arises from the study of ultrametric spaces.

Leinster proved that a finite ultrametric space has a nonnegative weighting. Furthermore, he established that the magnitude of an ultrametric space $X$ satisfies the bound $1 \leq \mathrm{Mag}(X) \leq e^{\mathrm{diam}(X)}$, where $\mathrm{diam}(X)$ is the diameter of $X$ \cite[Prop.2.4.18 and Cor.2.4.19]{Leinster2013}. This suggests that the existence of a nonnegative weighting plays a crucial role in controlling the analytic behavior of magnitude.

Motivated by these observations, we establish continuity for the class of finite metric spaces admitting a nonnegative weighting. By proving that the magnitude ratio is exponentially bounded by the Gromov-Hausdorff distance, we provide a natural generalization of the diameter bound for ultrametric spaces. 

\section{Preliminaries and Motivating Examples}
We begin by defining the primary class of metric spaces that will be the focus of our study.

\begin{definition}
\[
\mathrm{FMet}_{\geq 0} = \{X \mid X \text{ is a finite metric space with a nonnegative weighting}\}.
\]
\end{definition}

As mentioned in the introduction, an important family of spaces within this class is the ultrametric spaces.

\begin{example}
Any finite ultrametric space admits a nonnegative weighting, and therefore belongs to $\mathrm{FMet}_{\geq 0}$.
\end{example}

To clarify the scope of our study, we note that having a nonnegative weighting and being a positive definite space are two independent properties. The following examples illustrate this independence.

\begin{example}
Let $K_{n, m}$ be the graph with vertices $a_1, \dots, a_n$, $b_1, \dots, b_m$ and one edge between $a_i$ and $b_j$ for each $i$ and $j$.
\begin{enumerate}
    \item[(1)] $X =(\log\frac{3}{2})K_{3, 3}$ has a nonnegative weighting but is not positive definite. Indeed, all components of the weighting are $\frac{9}{35} >0$, while the zeta matrix $Z_X$ has an eigenvalue of $-\frac{1}{9}<0$.
    \item[(2)] $Y =(\log\frac{3}{2})K_{1, 3}$ is positive definite but does not have a nonnegative weighting. Indeed, the eigenvalues of $Z_Y$ are $\frac{5}{9}$(with multiplicity $2$) and $\frac{13 \pm 2\sqrt{31}}{9} > 0$, while one of the components of $Y$ has the weighting is $-\frac{1}{5} < 0$.
\end{enumerate}
\end{example}

\section{Continuity of Magnitude on $\mathrm{FMet}_{\geq 0}$}
Having established that $\mathrm{FMet}_{\geq 0}$ encompasses spaces that are not necessarily positive definite, we now investigate the analytic properties of magnitude within this class. The following theorem establishes a fundamental bound on the ratio of magnitudes in terms of the Gromov-Hausdorff distance. This bound generalizes the behavior observed in ultrametric spaces.

\begin{theorem}
Let $X, Y \in \mathrm{FMet}_{\geq 0}$. Then
\[
e^{-2d_{GH}(X, Y)} \leq \frac{\mathrm{Mag}(X)}{\mathrm{Mag}(Y)} \leq e^{2d_{GH}(X, Y)}.
\]

\begin{proof}
Fix $\eta > 0$. Take a metric $d$ on $X \sqcup Y$ that restricts to the given metrics on $X$ and $Y$ such that $\delta = d_H^d(X, Y) < d_{GH}(X, Y) + \eta$.

Let $f: X \to Y$ and $g: Y \to X$ be the mappings defined as follows:
For each $x \in X$, choose $f(x)$ such that $d(x, f(x)) \leq \delta$.
Similarly, for each $y \in Y$, choose $g(y)$ such that $d(y, g(y)) \leq \delta$.

Let $w_X$ and $w_Y$ be the weightings for $X$ and $Y$, respectively. Define $\tilde{w}_y$ for each $y \in Y$ as
\[
\tilde{w}_y = \sum_{x \in f^{-1}(y)} (w_X)_x \geq 0,
\]
where $\tilde{w}_y = 0$ if $f^{-1}(y) = \emptyset$.

By the triangle inequality, it follows that
\[
\forall x \in X, \forall y \in Y, \quad d_X(g(y), x) - 2\delta \leq d_Y(y, f(x)) \leq d_X(g(y), x) + 2\delta.
\]

Hence we have
\[
e^{-2\delta}e^{-d_X(g(y), x)} \leq e^{-d_Y(y, f(x))} \leq e^{2\delta}e^{-d_X(g(y), x)}.
\]

Multiplying by $(w_X)_x$ and summing over $x \in X$, we obtain
\[
e^{-2\delta}\sum_{x \in X} e^{-d_X(g(y), x)}(w_X)_x \leq \sum_{x \in X}e^{-d_Y(y, f(x))} (w_X)_x \leq e^{2\delta}\sum_{x \in X} e^{-d_X(g(y), x)}(w_X)_x.
\]

Since $w_X$ is a weighting,
\[
\sum_{x \in X}e^{-d_X(g(y), x)}(w_X)_x = 1,
\]
and note that
\[
\begin{aligned}
\sum_{x \in X} e^{-d_Y(y, f(x))}(w_X)_x 
&= \sum_{y' \in Y} \sum_{x \in f^{-1}(y')} e^{-d_Y(y, y')}(w_X)_x \\
&= \sum_{y' \in Y} e^{-d_Y(y, y')} \sum_{x \in f^{-1}(y')} (w_X)_x \\
&= (Z_Y \tilde{w})_y.
\end{aligned}
\]

Thus, the inequality simplifies to
\[
e^{-2\delta} \leq (Z_Y \tilde{w})_y \leq e^{2\delta}.
\]

Summing over $y \in Y$ against $(w_Y)_y$, we have
\[
e^{-2\delta} \mathrm{Mag}(Y)\leq \sum_{y \in Y}(w_Y)_y(Z_Y \tilde{w})_y \leq e^{2\delta} \mathrm{Mag}(Y).
\]

The middle term is

\[
\begin{aligned}
\sum_{y \in Y}(w_Y)_y(Z_Y \tilde{w})_y 
&= w_Y^T(Z_Y \tilde{w}) \\
&= (Z_Y w_Y)^T \tilde{w} \\
&= (1 \dots 1) \tilde{w} \\
&= \sum_{y \in Y} \tilde{w}_y \\
&= \sum_{y \in Y} \sum_{x \in f^{-1}(y)}(w_X)_x \\
&= \sum_{x \in X}(w_X)_x \\
&= \mathrm{Mag}(X).
\end{aligned}
\]

Hence, we obtain
\[
e^{-2\delta}\mathrm{Mag}(Y) \leq \mathrm{Mag}(X) \leq e^{2\delta}\mathrm{Mag}(Y).
\]

Since $\eta > 0$ is arbitrary and $\delta < d_{GH}(X, Y) + \eta$, we conclude that
\[
e^{-2d_{GH}(X, Y)} \leq \frac{\mathrm{Mag}(X)}{\mathrm{Mag}(Y)} \leq e^{2d_{GH}(X, Y)}.
\]
\end{proof}
\end{theorem}

An immediate consequence of this theorem can be obtained by taking the logarithm of both sides of the inequality, which yields:
\[
|\log \mathrm{Mag}(X) - \log \mathrm{Mag}(Y)| \leq 2d_{GH}(X, Y).
\]
This directly implies the following Lipschitz continuity.

\begin{corollary}
The mapping 
\[
\log \mathrm{Mag} \colon \mathrm{FMet}_{\geq 0} \to \mathbb{R}, \quad X \mapsto \log \mathrm{Mag}(X)
\]
is Lipschitz continuous.
\end{corollary}

The magnitude itself satisfies the following.

\begin{corollary}
The mapping 
\[
\mathrm{Mag} \colon \mathrm{FMet}_{\geq 0} \to \mathbb{R}, \quad X \mapsto \mathrm{Mag}(X)
\]
is locally Lipschitz continuous, but it is not Lipschitz continuous.
\begin{proof}
First, we will show that it is locally Lipschitz continuous.
By the preceding theorem, we have
\[
(e^{-2d_{GH}(X, Y)} - 1)\mathrm{Mag}(Y) \leq \mathrm{Mag}(X) - \mathrm{Mag}(Y) \leq (e^{2d_{GH}(X, Y)} - 1)\mathrm{Mag}(Y).
\]

Observe that $1 - e^{-t} \leq e^t - 1$ for all $t \geq 0$. Thus, it follows that
\begin{align*}
|\mathrm{Mag}(X) - \mathrm{Mag}(Y)| 
&\leq \mathrm{Mag}(Y)(e^{2d_{GH}(X, Y)} - 1) \\
&\leq 2\mathrm{Mag}(Y) e^{2d_{GH}(X, Y)} d_{GH}(X, Y),
\end{align*}
where the last inequality follows from the fact that $e^{2t} - 1 \leq 2te^{2t}$ for all $t \geq 0$.

Next, we will show that it is not Lipschitz continuous.
Let $A = \{*\}$ and $B =(\log(n-1))K_n$ for $n \geq 3$, where $K_n$ is a complete graph with $n$ vertices.
Since $\mathrm{Mag}(A) = 1$ and $\mathrm{Mag}(B) = \frac{n}{2}$, we have
\[
|\mathrm{Mag}(A) - \mathrm{Mag}(B)| = \frac{n}{2} - 1.
\]

Furthermore, since $d_{GH}(A, B) = \frac{\log(n-1)}{2}$, it follows that
\[
\frac{|\mathrm{Mag}(A) - \mathrm{Mag}(B)|}{d_{GH}(A, B)} = \frac{n - 2}{\log(n - 1)} \to \infty \quad \text{as } n \to \infty.
\]

Therefore, there does not exist a constant $L \geq 0$ such that 
\[
\frac{|\mathrm{Mag}(X) - \mathrm{Mag}(Y)|}{d_{GH}(X, Y)} \leq L
\]
for all $X, Y \in \mathrm{FMet}_{\geq 0}$.
\end{proof}
\end{corollary}

\textbf{Acknowledgements.}
I would like to thank Professor Masahiko Yoshinaga for his generous guidance and advice.

\end{document}